\documentclass[11pt]{article}
\usepackage{changes}
\usepackage{hyperref}
\usepackage{amsmath, amscd, amsfonts, amssymb, amsthm, tikz, times}
\usepackage[T1]{fontenc}
\newtheorem{thm}{Theorem}[section]

\newtheorem{lem}[thm]{Lemma}
\newtheorem{prop}[thm]{Proposition}
\theoremstyle{definition}
\newtheorem{defi}[thm]{Definition}

 \newcommand{\floor}[1]{\left\lfloor #1 \right\rfloor}
\def\1{{\mathchoice {1\mskip-4mu\mathrm l}      
{1\mskip-4mu\mathrm l}
{1\mskip-4.5mu\mathrm l} {1\mskip-5mu\mathrm l}}}
\newcommand{\ssup}[1] {{{\scriptscriptstyle{({#1}})}}}
\def\comment#1{}
\def\olra{\overleftrightarrow}
\def\ti{\to\infty}
\def\tr{\text{tr}}
\newcommand{\eps}{\varepsilon}
\def\parent#1{#1^{-1}}

\def\bi{\begin{itemize}}
\def\ei{\end{itemize}}

\def\bbE{\mathbb{E}}
\def\bbP{\mathbb{P}}

\def\eps{\epsilon}

\def\be{\begin{equation}}
\def\ee{\end{equation}}
\def\bea{\begin{eqnarray}}
\def\eea{\end{eqnarray}}
\def\nn{\nonumber}
\def\ff{\infty}
\def\({\left(}
\def\){\right)}
\def\[{\left[}
\def\]{\right]}

\def\lb{\left|}
\def\rb{\right|}

\def\la{\langle}
\def\ra{\rangle}

\def\ID{\text{ID}}

\def\Spec{\mu}

\def\emp{{\tiny\varnothing}}

\def\r{\mathcal{R}}
\def\cc{\mathcal{C}}
\def\kk{\mathcal{K}}

\def\ss{\mathcal{S}}

\def\Roo{{\R\backslash\{0\}}}

\def\Fcal{\mathcal{F}}

\def\Pil{\lambda_\Pi}

\newcommand{\N}{\mathbb{N}}

\def\N{\mathbb{N}}
\def\R{\mathbb{R}}
\def\Z{\mathbb{Z}}
\def\C{\mathbb{C}}

\def\L{\mathcal{L}}

\def \P{\mathbf{P}}
\def \E{\mathbb{E}}

\newcommand{\Hcal}{\mathcal H}

\def \CC{{\cal C}}
\def \DD{{\cal D}}
\def \EE{{\bf E}}

\def \GG{{\cal G}}

\def \TT{\mathbb{T}}

\begin{document}

\title{On the speed and spectrum of mean-field random walks among random conductances}

\author{Andrea Collevecchio$^{\,\rm 1}$ and Paul Jung$^{\,\rm 2}$\\
{\footnotesize{$^{\rm 1}$Department of Mathematical Sciences, Monash University, Clayton, VIC 3800, Australia}}\\
{\footnotesize{$^{\rm 2}$Department of Mathematical Sciences, KAIST, Daejeon, South Korea}}}

\maketitle \abstract{We study random walk among random conductance (RWRC) on complete graphs with $n$ vertices. The conductances are i.i.d. and the sum of  conductances emanating from a single vertex asymptotically has an infinitely divisible distribution corresponding to a L\'evy subordinator with infinite mass at $0$. We show that, under suitable conditions, the empirical spectral distribution of the random transition matrix associated to the RWRC converges weakly, as $n\to\ff$, to a symmetric deterministic measure on $[-1,1]$, in probability with respect to the randomness of the conductances.  In short time scales, the limiting underlying graph of the RWRC is a Poisson Weighted Infinite Tree, and we analyze the RWRC on this limiting tree. In particular, we show that the transient RWRC exhibits a phase transition in that it has positive or {weakly zero speed} when the mean of the largest conductance is finite or infinite, respectively. 
} \vspace{5mm}

\\

\section{Introduction}

In \cite{bordenave2011spectrum}, the limiting spectral distribution  of generators for Markov chains (Markov under a quenched measure) on randomly weighted complete graphs $\(G_n, n\in\N\)$ with $n$ vertices was studied where the weights were
interpreted as conductances across the edges.
Since models on complete graphs are considered mean-field models, the model of \cite{bordenave2011spectrum} is precisely a mean-field version of the so-called {\it random walk among random conductances} (RWRC) model or more succinctly, random conductance model.  See for instance, \cite{biskup2011recent} for an overview of the RWRC on $\Z^d$.

The model of \cite{bordenave2011spectrum} employs i.i.d. heavy-tailed positive weights, scaled by $n^{-1/\alpha}$, on the edges of the complete graph. Using this scaling, the conductances are in the domain of attraction of an $\alpha$-stable law as $n\to\infty$.
 For all $\alpha<2$,  viewing the edge-length as the inverse of the conductance on a given edge,  \cite{bordenave2011spectrum}  showed that the resulting weighted complete graphs converge, as $n\to\infty$, in a ``local weak sense'' (see \cite{benjamini2001recurrence, aldous2004objective})  to a version of Aldous'
Poisson Weighted Infinite Tree which is known more simply as the PWIT (this topology is coherent with physicists' cavity method which treats the scaling limit of complete graph as trees).

When couched in the mean-field i.i.d. setting, the phenomenon of convergence to a limiting infinite tree graph in fact requires the conductances to be heavy-tailed and scaled by $n^{-1/\alpha}$. Therefore, in this setting one may reasonably say that the scaling limit and thermodynamic limit of the finite graphs refer to the same limiting object. In particular, simply taking a thermodynamic limit without simultaneously scaling leads to a nonsense object.  On the other hand,  if one alternatively uses the absolute value of Gaussian weights (the natural extension to $\alpha=2$) as one's conductances, then the proper rescaling produces a degenerate graph in which all edge-lengths are infinite in the scaling limit. It was noted in \cite{ jung2018levy} that if one relaxes the i.i.d. requirement on conductances to the weaker condition of i.i.d. conductances for each fixed $n$, then one may obtain limiting graphs with conductances associated to infinitely divisible laws, rather than just $\alpha$-stable laws.  

In this work we analyze the limiting spectrum of the RWRC on sequences of finite weighted complete graphs whose local weak limit is a generalized PWIT. We also prove a phase transition in the speed of the transient RWRC on these generalized PWITs. {In particular, up to a mild assumption, we show that there is positive speed if the maximum conductance emanating from a given vertex has finite first moment and {weakly zero speed if this maximum conductance has infinite first moment. Here, `weakly' refers to the use of  $\liminf$ in place of a proper limit.} We will see that
the zero speed regime is reminiscent of the Bouchaud Trap Model (see \cite{bouchaud1992weak, benarous2006, fontes2008k}) except that we are trapped at an edge rather than a vertex. Our results on the speed of the RWRC can be compared to those of \cite{gantert2012random}. 

It should be noted that the scaling described in the previous paragraph gives a short time-scale result ($n$ goes to infinity, then $t$ goes to infinity) for the speed and, heuristically, a long time-scale result ($t$ goes to infinity, then $n$ goes to infinity) for the spectrum. 


The outline for the rest of the paper is as follows. We describe the model in the next section, and prove the convergence of the spectral distribution in the Section \ref{sec:spectrum}. Section \ref{sec:speed} covers the speed of the RWRC on the PWIT. 

\section{The Model}
Recall that an infinite divisibility probability measure $\mu$ associated to a subordinator, with no drift component, has a  L\'evy exponent $\Psi$ which is defined by
$$e^{\Psi(\theta)}:= \int_{\R} e^{i\theta x} \mu(dx) \quad\text{for } \theta\in\R.$$
We refer the reader to \cite{kyprianou2006introductory} or \cite{kallenberg2002foundations} for more details.

The  ``driftless'' infinitely divisible measure $\mu$ is supported on $(0,\infty)$ and has distribution $\ID(\Pi)$ whenever the exponent corresponds to a positive L\'evy measure
$\Pi$ and takes the form
 \begin{equation}\nn
\Psi(\theta):=\int_{(0,\infty)} (e^{i\theta {{} x}}-1)\, \Pi(dx)\,,
\end{equation}
where
$\Pi(dx)$ satisfies
\begin{equation}\label{subordinatormeas}
\int_{(0,\infty) } (1\wedge x)\,\Pi(dx)<\infty.
\end{equation}
We remark that this condition, for positive infinitely divisible  distributions, is different from the condition for general infinitely divisible  distributions which use the integral kernel $1\wedge x^2$ instead.  As is well known, the use of $1\wedge x$ here guarantees that the cumulative jumps of the associated subordinator, over finite time intervals, remains summable.   In addition to \eqref{subordinatormeas}, we will also assume that the L\'evy measure $\Pi$ is infinite.

{\bf Mean-field RWRCs (finitely many vertices):}
Let $G_n$ be the complete graph with vertex and edge sets $(V_n,E_n)$ where $V_n=\{1,\ldots,n\}$ and $E_n=\{e_{ij}, 1\le i<j\le n\}$.
The conductances on  the edges  $\{e_{ij}\}$ are i.i.d. and the conductance on edge $e_{ij}$ is denoted $\cc_n(i,j)=\cc_n(j,i)$.

We assume that conductances on edges adjacent a fixed vertex $i$ are positive and satisfy the following distributional limit property
\begin{equation}\label{summable condition}
\lim_{n\to\infty}\sum_{j=1}^{n} \cc_n(i,j)  \ \text{ is }\
\ID(\Pi) \text{ with }\|\Pi\|=\infty.
\end{equation}
As is well known, the infinite mass portion of $\Pi$  must be in a neighborhood around zero, and $\Pi$  must also satisfy \eqref{subordinatormeas}.
In particular, \eqref{summable condition} implies that for each $i$, $\{\cc_n(i,j), j\in\N\}$ asymptotically look like a Poisson point process with intensity $\Pi$.
We do not consider in this work,
the case where $\|\Pi\|<\infty$, but it should be noted that the scaling limits of the graphs in these cases are just weighted Galton-Watson trees (see the following subsection).

Using these conductances we see that the (random) Markov kernel defined by
\be\label{def:k}
\kk_n(i,j):=\cc_n(i,j)/\rho_n(i), \quad
\rho_n(i):=\sum_{j=1}^n \cc_n(i,j),
\ee
is reversible with  respect to the measure $\sum_{i\in V_n}\rho_n(i)\delta_i$ since
$$\rho_n(i)\kk_n(i,j)=\rho_n(j)\kk_n(j,i).$$

\subsection{Scaling limits of weighted complete graphs: PWITs}\label{sec:pwit}

Let us review the definition of an infinite graph with a  PWIT$(\Pil)$ distribution. Start with a
single root vertex $\emp$ with an infinite number of (first
generation) children indexed by $\N$. The weight on the edge to the
$k$th child  is the $k$th arrival $\r_\ff(\emp,k)$ (ordered from {smallest to biggest}) of a Poisson process on $\Roo$
with some intensity ${\lambda_{\Pi}}$. The weight represents the resistance across the edge and can also be thought of as a local distance function.
 Note that the resistances emanating from $\emp$ which are less than any fixed $\eps>0$ are independent, or equivalently, the conductances emanating from $\emp$ which are above any $\eps>0$ are independent.

In our situation the intensity $\Pil$
is derived from the L\'evy measure $\Pi$ on $(0,\infty)$ by inverting:
\be \label{pi} \Pil\{x:1/x\in B\}:=\Pi(B). \ee
For example, if
$\Pi(dx)$ is mutually absolutely continuous with respect to Lebesgue measure with density $f_\Pi(x)dx$ then
$\Pil(dx)$ is also absolutely continuous  {with respect to the Lebesgue measure and its  density is }
$x^{-2}f_\Pi(1/x)dx$ where $x^{-2}$ is the change-of-measure factor. {Since the infinite mass portion of $\Pi$ is near $0$, the infinite mass portion of $\Pil$ is near $\ff$; this indicates that most offspring of any given vertex will lie across edges of high resistance, in the sense that all but finitely many of these resistances will be more than any fixed positive value.}

If $G$ has a root at $\emp$ we write $G[\emp]$ for the rooted graph
with (random) weights assigned to each edge. Slightly
abusing notation, we denote the subgraph of a PWIT($\Pil$) formed by
the root $\emp$, its children, and the weighted edges in between, by
$\N[\emp]$.

\includegraphics[scale=.35]{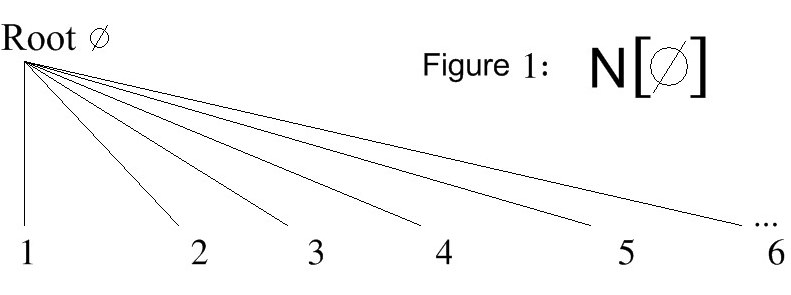}

We continue now with other generations. Every vertex $v$ in
generation $g\ge 1$ is given an infinite number of children indexed
by $\N$ forming the subgraph $\N[v]$. We denote the collection of all the $\N[v]$ for $v$ in generation $g-1$ by $\N^g$.  Thus the vertex set $V_\infty$ is
\be
\N^F:=\bigcup_{g\ge 0} \N^g     
\ee
where $\N^0=\emp$.
The weights on edges to children in generation $g+1$, from some
fixed vertex $v$ in generation $g$, are found by repeating the
procedure for the weights in the first generation, namely according
to the points of an independent Poisson random measure on $(0,\ff)$ with
intensity $\Pil(dx)$.

\vspace{3mm}
\includegraphics[scale=.35]{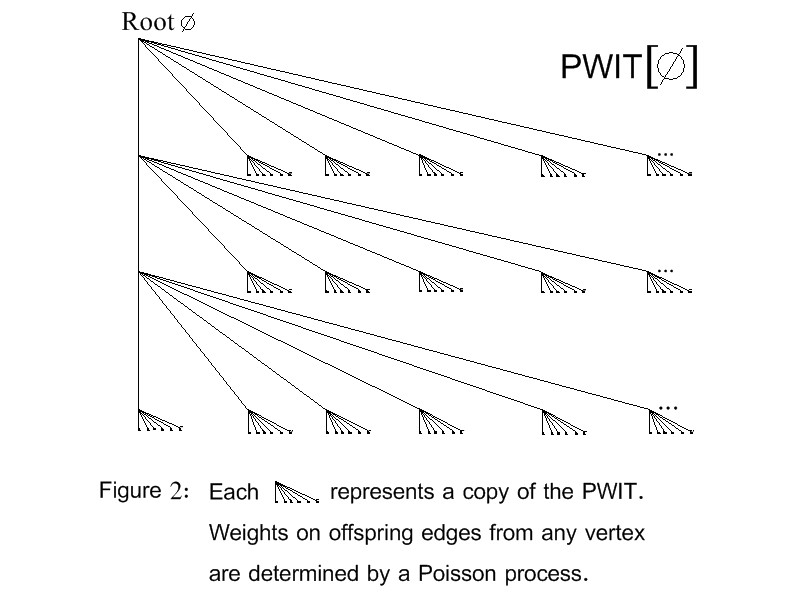}

If the resistance (or local distance function) on edge $e_{ij}$ of the finite graph $G_n[1]=(V_n,E_n)$ is taken to be $\r_n(i,j):=1/\cc_n(i,j)$ and the vertex $1$ is always chosen as the root of $G_n$, then
it is known that $(G_n[1], n\ge 1)$ converges in the local weak sense. In particular, this is convergence in the Benjamini-Schramm topology which turns the space of weighted rooted graphs into a Polish space so that convergence in distribution also makes sense.
 The limit of the sequence $(G_n[1], n\ge 1)$
 is a rooted random graph $G_\infty[\emp]$ which has a PWIT($\Pil$) distribution (see \cite[Prop 4.5]{jung2018levy}), and here $1$ is identified with $\emp$; moreover, since $\|\Pi\|=\ff$, the graph is infinite. As in \cite[Section 2.5]{bordenave2011spectrum} we will therefore think of $V_n$ as being (randomly) embedded in $V_\infty$ in such a way that the graphs converge in a local weak sense.
Note that the weights or resistances on edge $e_{vw}$ of $G_\infty[\emp]$ can be thought of as coming from some infinite conductance matrix
 \be\label{def:inf conductance matrix}
\cc_{\infty}(v,w) =\la \delta_v, \cc_{\infty}\delta_w \ra:=
\begin{cases}
& 1/\r_{\ff}(v,w)\quad \text{if } v\sim w\\
& 0 \quad\quad \quad \quad\quad\  \text{otherwise}.
\end{cases}
\ee

Before moving on let us give two important examples.

\noindent {\bf Examples:}
\begin{enumerate}
\item {\it $\alpha$-stable subordinator PWITs:} The  L\'evy measures take form  $C x^{-1-\alpha}\1_{x>0} dx$ where $\alpha\in(0,1)$, so the conductances $\{\cc_n(j,k)\}$ are in the domain of attraction of  a positive $\alpha$-stable distribution, i.e., a positive distribution such that the distribution function has tail $1-F(x)=L(x)x^{-\alpha}$ where $L$ is a slowly varying function and $\alpha\in(0,1)$. The matrix of conductances has entries which are i.i.d. up to the symmetry condition $\cc_n(j,k)=\cc_n(k,j)$. This special case was studied in \cite{bordenave2011spectrum}. This special case also corresponds to a mean-field distance model in dimension $d$ equal to $\alpha$, when one considers only nearest neighbors of the PWIT (see \cite[Sec. 4.1]{aldous2004objective}). 
\item {\it Tempered $\alpha$-stable subordinator PWITs:} The L\'evy measure is $C x^{-1-\alpha} e^{-x^p} \1_{x>0} dx$ where $\alpha\in(0,1), p>0$. This a tempered version of the first example, and when $p=1$ is a natural continuation of the Gamma$(\alpha,\beta)$-process (whose L\'evy measure is $\frac{\alpha^\beta}{\Gamma(\beta)}x^{\beta-1}e^{-\alpha x} \1_{x>0}  dx$) to negative values of $\beta$.  It corresponds to a class of  infinitely divisible distributions which has gained popularity since it provides good fits to actual data in the actuarial sciences, biostatistics, finance, and physics (see \cite{grabchak2016tempered}).
\end{enumerate}

{\bf Mean-field RWRCs (infinitely many vertices):}
We are now in a position to extend the definition in \eqref{def:k} to a  limiting Markov transition matrix $\kk_\infty$.  For its spectral analysis, we view $\kk_\infty$
as an operator on $\ell^2(V_\infty)$ where $V_\infty=\N^F$ is the vertex set of a PWIT($\Pil$). In particular, let $\DD\subset \ell^2(V_\ff)$ be the set of vectors
with finite support so that $\DD$ forms a core. For $v,w\in\N^F$ we define (using the convention that the second vector is a ``column'')
\be
\kk_\infty(v,w)=\langle\delta_v,\kk_\infty\delta_w\rangle:=
\begin{cases}
& \cc_{\infty}(v,w)/\rho_\infty(v)\quad \text{if } v\sim w\\
& 0 \quad\quad \quad \quad\quad\  \text{otherwise}
\end{cases}
\ee
and $\rho_\infty(v):=\sum_{w\in V_\infty} \cc_\ff(v,w)$ which is almost surely finite by condition \eqref{subordinatormeas}. In particular, if $v$ is the $k$th child of its parent in the construction of $G_\ff[\emp]$, then
the random variable $\rho_\infty(v)$ has a distribution equal to the convolution of ID$(\Pi)$ with the distribution of the $k$th arrival (this time ordered from biggest to smallest) of a Poisson process on $(0,\ff)$ with intensity $\Pi$. In the special case of the root, $\rho_\infty(\emp)$ simply has an ID$(\Pi)$ distribution.
 The random operator above is the transition matrix of a (random) Markov chain $(X_m)_{m\ge 0}$ on the limiting PWIT($\Pil$).

The operator $\kk_\infty$ is not symmetric, but becomes symmetric in the space $L^2(V_\ff, \rho_\ff)$ with inner product
\begin{equation}\label{eq:bdd op}
\langle \phi , \psi \rangle_\rho:=\sum_{v\in V_\ff}\rho_\ff(v)\phi(v)\psi(v).
\end{equation}
Under this inner product,  Cauchy-Schwarz shows that the operator norm is bounded by 1 and thus $\kk_\ff$ is also self-adjoint, 
\begin{align*}
\langle\kk_\infty\phi,\kk_\infty\phi\rangle_\rho&= \sum_{v\in V_\ff} \rho_\ff(v) \left|\sum_{w\in V_\ff} \kk_\ff(v,w) \phi(w)\right|^2\\
&\le \sum_{v\in V_\ff} \rho_\ff(v) \sum_{w\in V_\ff} \kk_\ff(v,w) \phi(w)^2\\
&=\sum_{w\in V_\ff} \rho_\ff(w) \phi(w)^2 = \langle\phi,\phi\rangle_\rho.
\end{align*}
\section{The limiting spectrum}\label{sec:spectrum}

Recall that the empirical
spectral distribution (ESD) of an $n\times n$ matrix $A_n$ is defined as
\be \label{esd}
\frac 1 n \sum_{j=1}^{n} \delta_{\lambda_j(A_n)}
\ee
where $\{\lambda_j(A_n)\}_{j=1}^{n}$ are the eigenvalues { of $A_n$}.
In this section we show that the arguments of \cite{bordenave2011spectrum} concerning the limiting empirical spectral distribution (LSD) of $\kk_\ff$ extend to the more general setting
of infinitely divisible mean-field RWRCs.

\begin{thm}[Limiting empirical spectral distribution of mean-field RWRCs]\label{thm:spectrum}
Suppose \eqref{summable condition} holds. Then there exists a symmetric deterministic measure $\mu_{\kk_\infty}$ supported on $[-1,1]$ depending only on $\Pi$ such that we have the following convergence in probability of random probability measures
\be\nn
\mu_{\kk_n}:=\frac{1}{n}\sum_{k=1}^n\delta_{\lambda_k(\kk_n)}\stackrel{n\to\infty}{\Longrightarrow} \mu_{\kk_\infty}.
\ee
where as usual, $\Longrightarrow$ denotes weak convergence of probability measures.
\end{thm}

One reason for being interested in the limiting spectral distribution $\mu_{\kk_\infty}$ is that it contains all information about the expected $r$-step return probabilities of the limiting RWRC starting from a fixed vertex. To see this, one just takes the large $n$ limit in the  expected $r$-step return probability of the process $(X_m)_{m\ge 0}$ on $G_n$, from a given vertex, which is equivalent to the expected $r$th moment of $\mu_{\kk_n}$:
\begin{equation}\label{eq:return probs}
\E\[\langle\delta_1, \kk_n^r\delta_1\rangle\]=\E\left[\frac{1}{n}\sum_{j=1}^n \langle\delta_j, \kk_n^r\delta_j\rangle\right]=\E\left[\int_{-1}^1 x^r \mu_{\kk_n}(dx)\right].
\end{equation}
The first equality above follows from exchangeability while the second is just the Spectral Theorem.

%

%
        

The proof of Theorem \ref{thm:spectrum} follows the strategy of \cite[Theorem 1.4]{bordenave2011spectrum} which proves Theorem \ref{thm:spectrum} in the special cases of $\Pi(dx)=x^{-1-\alpha}dx, 0<\alpha<1$, corresponding to the L\'evy measures of $\alpha$-stable subordinators. They employ the following notion of {local operator convergence} for which we identify $v\in V$ with $\delta_v\in\ell^2(V)$ (for more details and for the intuition behind the following definition, we refer the reader to \cite{bordenave2011spectrum}).\\
\begin{defi}
A sequence of bounded operators
$(A_n)_{n\ge 1}$ on $\ell^2(V)$ is said to {\it locally converge at} $v\in V$ (or at $\delta_v\in\ell^2(V)$) to a closed linear operator $A_\infty$ on $\ell^2(V)$  at $u\in V$ if for some core $\DD$ of $A_\infty$, there is a sequence of bijections $\sigma_n:V\to V$ such that $\sigma_n(v)=u$ and for all $\phi\in \DD$
$$\sigma_n^{-1}A_n\sigma_n\phi \ \stackrel{n\to\infty}{\longrightarrow} \ A_\infty\phi \quad\text{in } \ell^2(V).$$
\end{defi}
As in \cite{bordenave2011spectrum}, we will use local operator convergence
of the renormalized symmetric operators
\bea \label{def:s matrix}
\ss_n(v,w):=\sqrt{\frac{\rho_n(v)}{\rho_n(w)}}\kk_n(v,w)=\frac{\cc_n(v,w)}{\sqrt{\rho_n(v)\rho_n(w)}} \quad \text{for }n\in\N\cup\{\infty\}.
\eea
 For  $n=\infty$, note that the random operator $\ss_\ff(v,w)$ is actually a.s. bounded and thus self-adjoint in $\ell^2(V_n)$. 
To see this, by \eqref{eq:bdd op} and the remark following it,  we need only show that $\ss_\ff$ has the same spectrum as $\kk_\ff$. It suffices to consider the (random) spaces $L^2(V_n, \rho_n)$, for all $n\in\N\cup\{\infty\}$,
with inner products 
\begin{equation*}
\langle \phi , \psi \rangle_\rho:=\sum_{v\in V_n}\rho_n(v)\phi(v)\psi(v),
\end{equation*}
and to realize that the map $\phi\mapsto\hat{\phi}$
with $$\hat{\phi}:=(\phi(1) \sqrt{\rho_n(1)}, \ldots, \phi(n) \sqrt{\rho_n(n)})$$
is almost surely an isometry from  $L^2(V_n,\rho_n)$ to $\ell^2(V_n)$.

We will also consider the $2n\times 2n$ matrices $\{\ss_n\oplus\ss_n, n\ge1\}$ and the direct sum operator
$\ss_\infty\oplus\ss_\infty'$ where $\ss_\infty$ and $\ss_\infty'$ are independent realizations of operators associated to a PWIT($\Pil$) through \eqref{def:inf conductance matrix} and \eqref{def:s matrix}. The matrices and the  operator are viewed as operating on a common space $\ell^2(V_\infty)\oplus \ell^2(V_\infty)$ by setting $\ss_n(v,w)=0$ whenever $\rho_n(v)=0$ (recall from the discussion above \eqref{def:inf conductance matrix} that $V_n$ is thought of as being randomly embedded in $V_\infty$).

The main tool used in the proof is the following lemma which is a generalization of Theorems 2.3 (iii) and 2.8 (iii) in \cite{bordenave2011spectrum}.

\begin{lem}\label{lem:operator conv}
Suppose \eqref{summable condition} holds. Then the operator sequence $(\ss_n)_{n\ge 1}$ locally converges, in distribution, at the vector $\delta_1\in \ell^2(V_\infty)$ to
$\ss_\infty$ at $\delta_\emp$. Moreover,  $(\ss_n\oplus \ss_n)_{n\ge 1}$ locally converges at the vector $(\delta_1,\delta_2)\in \ell^2(V_\infty)\oplus \ell^2(V_\infty)$ to
$\ss_\infty\oplus\ss_\infty'$ at $(\delta_\emp,\delta_\emp)$.
\end{lem}

The idea behind the above crucial lemma is that the random graph sequence $(G_n[1])_{n\ge 1}$ associated to the operator sequence $(\ss_n)_{n\ge 1}$ converges in the local weak sense as mentioned above \eqref{def:inf conductance matrix}, and this local weak convergence translates into local convergence of the operators at the root vertex.
We omit the proof of the lemma since it follows that of Theorems 2.3 (iii) and 2.8 (iii) in \cite{bordenave2011spectrum} (see also Section 4 in \cite{jung2018levy}).

\begin{proof}[Proof of Theorem \ref{thm:spectrum}]
The proof is essentially the same as that of Theorem 1.4 in \cite{bordenave2011spectrum}.  Consider the resolvents
$$R_z^{(n)}=(\ss_n-zI_n)^{-1} \quad \text{for }n\in\N\cup\{\infty\}.$$
By the first statement in Lemma \ref{lem:operator conv} concerning the local convergence in distribution, of $(\ss_n)_{n\ge 1}$ at $\delta_v$, and Skorokhod's Representation Theorem, there is a probability space on which $(\ss_n)_{n\ge 1}$ locally converges at $\delta_v$ for each $v\in V_\ff$, almost surely. Let us for the time being work on this probability space (in order to use dominated convergence below) so that, 
by Theorem VIII.25(a) in \cite{reed1980methods}, we have convergence of $(\ss_n)_{n\ge 1}$ in the strong resolvent sense, almost surely.

Recall that the Stieltjies transform of a measure on $\R$ is defined as
\be
\label{def:Stieltjies} s_\mu(z):=
\int_{\R} \frac{\mu(dx)}{x-z}, \ \ \ z\in\C\backslash\R.
\ee
If we denote the ESD of $\ss_n$ (equal to that  of $\kk_n$ by the argument below \eqref{def:s matrix}) by  $\mu_{\ss_n}$,  then by Fubini's Theorem $$s_{\E\Spec_{\ss_n}}(z) = \E\[
s_{\Spec_{\ss_n}}(z)\].$$
Also, by exchangeability
\be
\label{eq:resolv identity} \E\[
s_{\Spec_{\ss_n}}(z)\]= \frac 1 {n} \E \[\tr (\ss_{n}-zI)^{-1}\]=\E\[
R_z^{(n)}(1,1)\].
\ee
Using$\frac{1}{|x-z|}\le \frac{1}{|\text{Im}(z)|}$, one can bound the modulus of the diagonal of the Green's function
$$|R_z^{(n)}(j,j)|\le |\text{Im} (z)|^{-1},\quad z\in\C\backslash\R.$$
Now by dominated convergence, and the strong resolvent convergence of $(\ss_n)_{n\ge 1}$ as applied to the function $\delta_\emp\in\ell(V_\infty)$ (note that the set of all Kronecker delta functions form a core for this space), we can take the limit on the right side of \eqref{eq:resolv identity} to see that
 $s_{\E\mu_{\ss_n}}(z)$ converges, 
  for all $z\in\C^+$, to
\be\label{eq:resolv1}
\E\[\langle \delta_\emp, (\ss_\infty-zI)^{-1}\delta_\emp \rangle\]
\ee
which is the limit of the right side of \eqref{eq:resolv identity}.

Recall \cite[Sec. VII.2 and VIII.3]{reed1980methods} that the spectral measure $\mu_{\emp}$ of the self-adjoint operator $\ss_\infty$  associated to the vector $\varphi$
is defined by the relation
$$
\langle \varphi, f(\ss_\infty)\varphi\rangle =:\int_\R f(x) \mu_\varphi(dx), \quad \text{for bounded continuous } f.
$$
Using this notion, we have that \eqref{eq:resolv1} is also
equal to the Stieltjes transform of  
 the expected spectral measure $\E\mu_{\emp}$ associated to $\delta_\emp$.  Since $\delta_\emp$ has norm one, this expected spectral measure is
in fact a probability measure.
Now, by Theorem 2.4.4 of \cite{anderson2010introduction}, convergence of the Stieltjes transforms of  $(\E\mu_{\ss_n})_{n\ge 1}$  implies weak convergence of
the sequence $(\E\mu_{\ss_n})_{n\ge 1}$ to the probability measure $\E\mu_{\emp}$. This proves convergence in expectation of the ESDs (see \cite[pg. 135]{tao2012topics} for a definition of this convergence).

To improve the convergence in expectation to convergence in probability, we need only
show that for all $z\in\C^+$
\be \label{eq:l1conv of esd}
\lim_{n\to\infty}\E\[|s_{\mu_{\ss_n}}(z)-s_{\E\mu_{\emp}}(z)|\]=0.
\ee
Note that
$$\E\[|s_{\mu_{\ss_n}}(z)-s_{\E\mu_{\emp}}(z)|\]\le\E\[|s_{\mu_{\ss_n}}(z)-s_{\E\mu_{\ss_n}}(z)|\]+|s_{\E\mu_{\ss_n}}(z)-s_{\E\mu_{\emp}}(z)|$$
and that the second term goes to zero for all $z\in\C^+$, as $n\to\infty$, by convergence in expectation.
The first term on the right side equals
$$\E\[\lb \frac 1 n \sum_{k=1}^n\[R^{(n)}_z(k,k)-\E R^{(n)}_z(k,k)\] \rb\].$$
By exchangeability,
\bea
\nn&&\E\[\lb \frac 1 n \sum_{k=1}^n\[R^{(n)}_z(k,k)-\E R^{(n)}_z(k,k)\] \rb^2\]\\
\nn&=&\frac 1 n\E\[\lb R^{(n)}_z(1,1)-\E\[ R^{(n)}_z(1,1)\]\rb^2\]+\frac{n(n-1)}{n^2}\E\[\(R^{(n)}_z(1,1)-\E R^{(n)}_z(1,1)\)\(R^{(n)}_z(2,2)-\E R^{(n)}_z(2,2)\)\]\\
\nn&\le&\frac{1}{n (\text{Im}(z))^2}+\frac{n(n-1)}{n^2}\E\[\(R^{(n)}_z(1,1)-\E R^{(n)}_z(1,1)\)\(R^{(n)}_z(2,2)-\E R^{(n)}_z(2,2)\)\].
\eea
By Lemma \ref{lem:operator conv}, $R^{(n)}_z(1,1)$ and $R^{(n)}_z(2,2)$ are asymptotically independent, due to the fact that $\ss_\infty$ and  $\ss_\infty'$ are independent (see the paragraph preceding the lemma).
Since these random variables are bounded, for fixed $z$, they are also asymptotically uncorrelated
so that \eqref{eq:l1conv of esd} follows.
We have thus shown the existence of a unique limiting measure $\mu_{\kk_\ff}$ in probability. 

Let us now show that $\mu_{\kk_\ff}$ is symmetric.  Using \eqref{eq:return probs}, we see that the $r$th moments of $\mu_{\kk_n}$
represent the ``random'' $r$-step return probability (``random'' due to the fact that the graph is random), 
which is the probability of starting from the root and returning to the root in $r$ steps.  Since $G_n[1]$ converges in the local weak sense to a tree, all odd moments of $\mu_{\kk_n}$ must vanish in probability, as $n\to\ff$, since one can never return to a given vertex on a tree in an odd number of steps.  Since the measures $\mu_{\kk_n}$ are all supported on $[-1,1]$, this implies symmetry of the limiting measure. 
\end{proof}

\section{Speed of the RWRC}\label{sec:speed}
Let $\mathbf{X} = (X_m)_{m \in \N}$ denote the RWRC defined on the PWIT.
In this section, for each vertex $v$ (or random vertex $X_m$) denote by $|v|$ (resp. $|X_m|$) its graph-distance from the root, that is the number of edges along the shortest path connecting $v$ to $\emp$. 
Denote by $\bf{P}_\omega$ the quenched measure, that is we fix the environment {$\omega$}  (i.e. the random conductances).  {Let $\bbP$ be the annealed measure, which is the semiproduct $\bf{P} \times \bf{P}_\omega$, where $\bf{P}$ is the measure describing the environment.} 

	 The process $\mathbf{X}$ is transient-- a fact which can be deduced  from  Proposition 2.1 in \cite{gantert2012random}. In order to apply that result, we reason as follows. Perform a percolation on the PWIT, where we delete the edges whose conductances are smaller than $\epsilon>0$ to be specified below. This percolation is supercritical, if $\epsilon$ is chosen small enough. Hence there exists an  infinite connected component, which is a subtree, which we denote by $\L$ (choosing one of the infinite components arbitrarily).   {If the process $\mathbf{X}$ never reaches $\L$ then it is easy to see that it is transient, and we have nothing to prove.  Suppose it reaches $\L$.} {Define the process $\mathbf{X}^{(\L)}$  to be the restriction of $\bf{X}$ observed only} when it takes  steps in $\L$, which may be finite or even empty set of steps. In the case $\mathbf{X}^{(\L)}$ has an infinite number of steps, Proposition 2.1 in \cite{gantert2012random} establishes that $\mathbf{X}^{(\L)}$ is transient implying transience of $\mathbf{X}$. Moreover, it is clear that $\mathbf{X}$ cannot be recurrent if $\mathbf{X}^{(\L)}$ consists of only a finite number of steps since it then visits each vertex in $\L$ finitely often.

We are now ready to state the main results of this section.

\begin{thm}[Positive speed under finite mean of the largest conductance]\label{finsecm} 
Let $\tilde\cc$ be the largest conductance  associated to an edge connecting the vertex $\emp$ to one of its offspring. Assume that  $\bbE[\tilde \cc]<\infty$ and $\bbE[(\tilde \cc)^{-1}]<\infty$. There exists a constant $s \in (0, 1]$  such that for any $\omega$ belonging to  a set of $\bf{P}$-measure one, we have 
$$ {\bf P}_\omega \Big(\lim_{m\to\ff} \frac{|X_m|}m =  s\Big)=1.$$
\end{thm}
Our proof of Theorem \ref{finsecm}  first proves existence of the speed and then its positivity. For the existence proof we partly use arguments from \cite{gantert2012random} and \cite{lyons1995ergodic} (see also \cite{LP:book}). However, it seems difficult to adapt the argument from \cite{gantert2012random} with respect to the positivity of the speed\footnote{{In particular, in our model, the sequence of ``slabs'' formed by the regeneration points are not independent since the conductance on an edge which connects any two slabs creates a dependence structure. An additional complication in our setting is the fact that each vertex has infinitely many children.}}. Therefore, for positivity of the speed we use an argument inspired by work of Aid{\'e}kon, see \cite{aidekon2008transient}. This latter argument also gives information about the \lq rate of convergence\rq of the rescaled hitting times to the reciprocal of the speed, as stated in our next theorem.

		For $n \in \N$, let the hitting time of level $n$ be denoted  by $$T(n) := \inf \{ j \ge 0 \colon |X_j| =n\}$$ and  for any vertex $v$ set $$T_{v} := \inf \{ j \ge 0 \colon X_j =v\}.$$ For any vertex $v$ denote by $\parent{v}$ its parent. 

\begin{thm}[Limit theorems for hitting times]\label{Hittimes} 
Let $\tilde\cc$ be as above. Assume that $\bbE[\tilde \cc^{p}]<\infty$ and $\bbE[(\tilde \cc)^{-p}]<\infty$ for some $p >1$. {Then, for any $q \in (1, p)$, we have that 
$$\lim_{n \to \infty}  \E\left[\left(\frac{T(n)}{n} - \frac1 s \right)^q\right] = 0 .$$}

In order to prove a phase transition in the speed, we finally have that when the largest conductance has infinite mean, it acts like a trap for the random walk giving it weakly zero speed.

\end{thm}

\begin{thm}[Weakly zero speed under infinite mean of the largest conductance]\label{thinf} Let $\tilde{\cc}$ be as above.
 Assume that $\bbE[\tilde \cc] = \infty$ then, {\bf P}-a.s.,
 $$ {\bf P}_\omega \Big(\liminf_{m \to \infty} \frac{|X_m|}{m} = 0\Big) =1.$$
 \end{thm}

  \noindent{\bf Remarks:}
  \begin{enumerate}

\item The condition $\bbE[(\tilde \cc)^{-1}]<\infty$ in Theorem \ref{finsecm} follows if 
\begin{equation*} \label{}
\bbE[(\tilde \cc)^{-1}] = \int_{0}^{\ff} \bbP( (\tilde \cc)^{-1} > x) = \int_0^{\infty} (1/x^2) \bbP( \tilde \cc < x) dx < \infty.
\end{equation*}
The conductances coming from a given vertex form a Poison process with intensity $\Pi(dx)$, so $\bbP( \tilde \cc < x) = \exp(-\Pi(x,\ff))$, and thus $\bbE[(\tilde \cc)^{-1}] < \infty$ if, and only if,
\begin{equation*} \label{}
\int_0^{\infty} (1/x^2) \exp(-\Pi(x,\ff)) dx < \ff.
\end{equation*}
 \item Similarly, one has $\bbE[\tilde \cc]<\infty$ if and only if $\int_0^\ff x \, \Pi(dx)<\ff$ since
$$\bbE[\tilde \cc] = \int_{0}^{\ff} 1-\exp(-\Pi(x,\ff))\, dx $$
and the right side is finite if and only if  $\int_0^\ff x \, \Pi(dx)<\ff$. 

\item Clearly, the first example in Section \ref{sec:pwit} satisfies the conditions of Theorem \ref{thinf} while the second example satisifes the conditions of Theorem \ref{finsecm}.  One can interpret this as  {\it traps} being removed in the $\alpha$-stable subordinator PWIT as a result of ``tempering'' the largest conductances.

\item
Recalling the earlier description of the PWIT, the edge-conductances to offspring of a given vertex $v$ are given by the arrivals of a Poisson process on $(0,\ff)$ with intensity measure $\Pi$. Thus, conditioned on being larger than $\epsilon$, these conductances are independent (there are finitely many by \eqref{subordinatormeas}).  This is an important property that allows us to connect our model to previously used proof techniques for random walks in random environments.  
  \end{enumerate}

\subsection{Existence of the speed in the finite mean regime}

We use the ergodic theorem to show existence of the speed, and a key part of this argument is finding a reversible probability measure for the environment as observed from the random walker. Our reversible measure is motivated by a similar measure in \cite{gantert2012random}, and it only exists under the finite mean condition $\bbE[\tilde \cc]<\infty$ which partly explains the dichotomy between Theorems \ref{finsecm} and \ref{thinf}. 

Recall that $G_\ff[v]$ is our tree with random edge-weights, rooted at $v$. The Markov operator $\kk_\ff$ extends in a natural way to a process which includes the environment observed by the walker. In particular, such a process on the space of weighted rooted trees, $\GG_{\star}$, is given by the transition operator
\begin{align*}
K f(G_\ff[v]):= \sum_{w:w\sim v} \frac{\cc_\ff(v,w)}{\rho_\ff(v)} f(G_\ff[w])
\end{align*}
for an appropriate class of functions.
If we size-bias $\P$ to get the probability measure $\P_{sb}$ on $\GG_\star$ with corresponding expectation
$$\EE_{sb}[f(G_\ff[v])]:=\frac{\EE\big[\rho_\ff(v) f(G_\ff[v])\big]}{\EE[\rho_\ff(v)]},$$
then it turns out that $K$ is reversible with respect to $\P_{sb}$. 

To prove reversibility, we use the involution invariance or mass transport principle for  $G_\ff[v]$ with respect to an infinite measure on $\GG_\star\times V$ (see \cite[Sec. 5]{aldous2004objective} or \cite[Ex. 9.7]{aldous2007processes}). In particular consider $\GG_\star\times V$ to be the set of weighted rooted trees with a {\it distinguished directed edge} (from the root), and use the infinite measure on $\GG_\star\times V$ which has marginal $\P$ on $\GG_\star$ and counting measure on the edges emanating from the root, $V$. By \cite[Sec. 5]{aldous2004objective}, we have for $f$ and $g$ in $L_2(\P)$
\begin{align*}
 \EE\[ \sum_{w:w\sim \emp} f(G_\ff[\emp]) \cc_\ff(\emp,w) g(G_\ff[w])\] =\EE\[ \sum_{w:w\sim \emp}  g(G_\ff[\emp]) \cc_\ff(\emp,w) f(G_\ff[w])\].
\end{align*}
The reversibility of $\P_{sb}$ follows since for $f$ and $g$ in $L_2(\P)$
\begin{align}\label{eq:reversible}
\big\langle f(G_\ff[\emp]), \, K g(G_\ff[\emp])]\big\rangle_{sb}&:=\EE_{sb}[f(G_\ff[\emp]) K g(G_\ff[\emp])]\\
&=\frac{1}{\EE[\rho_\ff(\emp)]}\EE\left[\rho_\ff(\emp) f(G_\ff[\emp]) \sum_{w:w\sim \emp} \frac{\cc_\ff(\emp,w)}{\rho_\ff(\emp)} g(G_\ff[w])\right]\nn\\
&=\frac{1}{\EE[\rho_\ff(\emp)]}\EE\left[ f(G_\ff[\emp]) \sum_{w:w\sim \emp} \cc_\ff(\emp,w) g(G_\ff[w])\right]\nn\\
&=\frac{1}{\EE[\rho_\ff(\emp)]}\EE\left[ \sum_{w:w\sim \emp}  f(G_\ff[w]) \cc_\ff(\emp,w) g(G_\ff[\emp])\right]\nn\\
&=\big\langle K f(G_\ff[\emp]), \, g(G_\ff[\emp])]\big\rangle_{sb}.\nn
\end{align}

 \begin{prop}[Existence of the speed, finite speed regime] If  $\bbE[\tilde \cc]<\infty$, or equivalently \mbox{$\int_0^\ff x \, \Pi(dx)<\ff$,} then
 	$\lim_{m\to\infty} |X_m|/m$  exists, a.s, and is constant with respect to the random environment.
 \end{prop}
\begin{proof}
	The first part of the proof adapts a method of \cite{gantert2012random} based on \cite{lyons1995ergodic} (see also \cite{LP:book}). 
	
	Define bi-infinite random walk paths on the PWIT, $\TT\equiv G_\infty[\emp]$, under the measure $\P_{sb}$, by gluing together two independent walks starting from the root, and for each realization of $\TT$ let the probability measure governing these walks be denoted $RW$. For each $\omega\in RW$ this gives us a sequence of (not necessarily unique) vertices $\ldots, x_{-1}, x_0, x_1, \ldots =:\olra{x}$. 
	The collection of such paths is denoted by $\olra{\TT}$.
	 The paths are coupled to the PWIT (rooted at $x_0$) they are contained in, and the resulting path bundle over the space of trees is
	$${\rm PathsInTrees}:=\{(\olra{x},\TT):\olra{x}\in\olra{\TT}\}$$
	with probability measure $RW\times \P_{sb}$.
	 Define the shift operator $S$ by
	 $$(S\olra{x})_n:= x_{n+1},\quad S(\olra{x},\TT):=(S\olra{x},\TT)$$
	 with $S^k$ being the $k$th iteration of the shift.
	
	By the reversibility described in \eqref{eq:reversible}, 
the Markov chain on ${\rm PathsInTrees}$ whose transitions are induced by the bi-infinite random walk which shifts the path (thus also changing the root of the tree) is stationary. 
	
	Since $\mathbf{X}$ is transient, it is not hard to see that, for almost every $\omega\in RW$, $\olra{x}$ converges to one end of the PWIT, call it $\mathfrak{a}$, as $n\to\infty$ and to a different end, $\mathfrak{b}$, as $n\to -\infty$. Consider the random variable $Y$ which has value 1 if $x_1$ is closer to $\mathfrak{a}$ than $x_0$, value -1 if $x_1$ is closer to $\mathfrak{b}$ than $x_0$, and value 0 otherwise.
	By the Ergodic Theorem
	$$\frac1n \sum_{i=1}^n Y(S^i(\olra{x},\TT))$$
	converges a.s. to some random variable. But note that this limiting random variable is simply the speed $s$.
	
	Next we prove that $s$ is a.s. constant. For any vertex $v$, denote by $\mathcal{G}(v)$ the PWIT which contains  $v$, its descendants, $\parent{v}$, and the edges connecting them. The conductances assigned to each edge are the same as the original PWIT, with the exception of the conductance  assigned to the edge connecting $v$ to its parent, which is set to some constant $M>0$. We say that $\mathcal{G}(v)$ is {\bf good} if the restriction of $\mathbf{X}$ to its graph never returns to  $\parent v$ after its first visit to $v$. 

		Define 
			$$ H(n) := \inf \left\{k \colon  |X_k| \ge n,\mbox{ and }  \cc(X_k, \parent{X_k})\le M \right\}.$$Notice that $s$ is measurable with respect to the $\sigma$-algebra generated by 
		$$ \bigcup_{n} \Fcal_{H(n)}.$$
		On the other hand, by transience, the probability that after time $H(n)$  the process never goes back to the parent of $X_{H(n)}$ is bounded away from zero. If $\mathcal{G}(X_{H(n)})$ is good, then the paths $(X_{t})_{t \ge H(n)}$ and $(X_t)_{t < H(n)}$ are disjoint and they are conditionally  independent, given $(H(n), X_{H(n)})$. Hence $s$ is independent of $\Fcal_{H(k)}$ for any fixed  $k$ and must therefore be constant.
	
%

\end{proof}


\subsection{Proof of Theorem~\ref{finsecm}}


We need several preliminary results to prove positivity of the speed.
Define $\omega$ to be the environment of the tree. Recall that $T_{v}$ is the hitting time of  $v$ by the process  $\mathbf{X}= (X_m)_{m \ge 0}$. Define
$$ \beta_{\omega}(v) := \P_{\omega}( T_{\parent{v}}= \infty\;|\; X_0 = v).$$
For any pair of  distinct vertices $v$ and $u$, denote by $[v, u]$ the self-avoiding  path connecting $v$ to $u$. We say that $v $ is an ancestor of $u$,  if $v \neq u$ and  if $v$ lies in $[\emp, u]$. 

\begin{prop}\label{impro}
Let $\tilde\cc$ be the largest conductance  associated to an edge connecting the vertex $\emp$ to one of its offspring.
Suppose that for some $p\ge 1$, $\bbE[\tilde \cc^p] <\infty$ and $\bbE[(\tilde \cc)^{-p}] <\infty$. For any vertex $x_0 \neq \emp$,  we have
	\begin{equation}\label{rtt0CLT}
\bbE\left[\left(\frac 1{\beta_{{\omega}}(x_0)}\right)^{p} \right]< \infty.
\end{equation}

\end{prop}
\begin{proof}

This proof follows the same strategy as the proof of  Lemma 2.2 in \cite{aidekon2008transient}, with modifications due to the nature of the environment that we consider.  
For each vertex $y$,  with $|y| \ge 2$, define

$$A(y)\equiv A_\omega(y) := \frac{\cc(y^{-1}, y)}{\cc(y^{-2}, \parent{y})},$$
where $y^{-2}$ is the parent of $\parent{y}$. For a generic vertex $x$, {with $|x|\ge 1$},
denote by $y_{i}$, $ i \in \N$,  the offspring of $x$. We next prove the relation
\begin{equation}\label{adk1}
\beta_{{\omega}}(x) = \sum_{i=1}^\infty \frac{A(y_i)}{1+ \sum_{j=1}^\infty A(y_j)} \beta_{{\omega}}(y_i) + \sum_{i=1}^\infty \frac{A(y_i)}{1+ \sum_{j=1}^\infty A(y_j)}\Big(1-\beta_{{\omega}}(y_i)\Big) \beta_{{\omega}}(x).
\end{equation}
In order to prove \eqref{adk1}, notice that
$$\frac{A(y_i)}{1+ \sum_{j=1}^\infty A(y_j)}$$
is the probability to jump to $y_i$ from $x$.  Given that the process jumps to $y_i$, it either goes back to $x$ in the future or does not. In the former case we restart our reasoning using the Markov property of the quenched process.
Rearranging \eqref{adk1}, we get
\begin{equation}\label{aidk2}
  \frac{1}{\beta_{{\omega}}(x)}\, =\,  1 + \frac{1}{\sum_{i=1}^\infty A(y_i) \beta_{{\omega}}(y_i)} \le 1 + \min_{i} \frac 1{A(y_i) \beta_{{\omega}}(y_i)}.
  \end{equation}
  
  Define a particular path, depending on $\omega$, in the following way.
  The first vertex of the path is  $x_0$  that was fixed in the statement of the proposition.  We set $x_1, x_2, \ldots$ recursively by setting  $x_{n+1}$ to be the child of $x_n$ which has maximal conductance, i.e. which maximizes $u \mapsto \cc(x_n, u)$ where $u\neq x_n^{-1}= x_{n-1}$.
Define $\Hcal_k\equiv \Hcal_k(\omega) $ to be the set of offspring of $x_k$ which are different from $x_{k+1}$, for $k\ge 0$.  Fix a constant $C$ large enough, to be specified later. For all $n\in \N$, 
define 
\begin{equation}\label{eq:en}
E_n := \left\{\omega: \forall k \in \{2, 3, \ldots, n-2\}, \forall z \in \Hcal_k(\omega) \mbox{ we have } \Big(A_\omega(z) \beta_{\omega}(z)\Big)^{-1} >C\right\}.
\end{equation}
We set $E_0^c = \emptyset$. Notice that  $E_{n+1} \subset E_n$ and that on the event $ E^c_{n+1}\cap E_n$ we have
$$ \min_{y \in  \Hcal_n}  \frac {1 }{A(y)\beta_{{\omega}}(y)} \le C.$$
Combining these two facts with  \eqref{aidk2},
  we infer the following.
\begin{equation}\label{ven}
\begin{aligned}
\frac{\1_{E_n}}{\beta_{{\omega}}(x_n)} &
\le  1 + \min_{y \in  \Hcal_n} \frac {\1_{E_{n+1}^c} \1_{E_n} }{A(y)\beta_{{\omega}}(y)}  +\min_{y \in  \Hcal_n} \frac {\1_{E_{n+1}}}{A(y)\beta_{{\omega}}(y)}\\
&\le  1 + C  + \min_{y \in  \Hcal_n} \frac {\1_{E_{n+1}}}{A(y) \beta_{{\omega}}(y)}\le 1+ C + \frac{\1_{E_{n+1}}}{A(x_{n+1}) \beta_{{\omega}}(x_{n+1})}.
\end{aligned}
\end{equation}
Consider two distinct vertices, $z, v$, and let $(z_i)_{i\in \N}$ and $(v_i)_{i \in \N}$, respectively, be the offspring of $z$ and $v$. Recall that $(\cc(z, z_i))_{i \in \N}$ and $(\cc(v, v_i))_{i \in \N}$ are independent. This fact implies  that $\cc(x_n, x_{n+1})$ are i.i.d. (recall that $x_n$ and $x_{n+1}$ are random).  In turn, this implies that the  process $(A(x_n), n\ge 1)$ is one-dependent, in other words $A(x_n)$ and $A(x_j)$ are independent if the vertices $x_n$ and  $ x_j$ are not neighbors.
 Set $$B(n) := \1_{E_n} \prod_{k=1}^n \frac 1{A(x_k)}= \1_{E_n} \frac{ \cc(x_0^{-2}, \parent{x_0})}{\cc(x_n, x_{n-1})}.$$

 By iterating the equation \eqref{ven} we have 

\begin{align}\label{finstforCLT}
\left(\frac 1{\beta_{{\omega}}(x_0)}\right)^{p } &\le {(1+C)^{p} \left(\sum_{n \ge 0} B(n) \right)^{p}}\\
&{= (1+C)^{p} \left(\sum_{n \ge 0} 2^{n+1} B(n) 2^{-n-1} \right)^{p}\nn} \\
&\le (1+C)^{p}{2^{p}} \sum_{n \ge 0} \left(2^{n}B(n)\right)^{p} 2^{-n-1}\nn \\
&= (1+C)^{p}{2^{p-1}} \sum_{n \ge 0}  2^{(p-1) n}\left(B(n)\right)^{p},\nn
\end{align}
where in the second last step we used Jensen's inequality with respect to the probability measure that assigns probability $2^{-n -1}$ to $n \in 0 \cup \N$.
Notice that
$$ \bbE\left[\(B(n)\)^p\right] = \bbP(E_n) \bbE\left[\cc^p(x_0^{-2}, \parent{x_0})\right]  \bbE\left[\frac {1} {\cc^p(x_n, x_{n-1})}\right]<\infty,$$
since  $\bbE[\tilde \cc^p] <\infty$ and $\bbE[(\tilde \cc)^{-p}] <\infty$, and since for $n\ge1$ we have $E_n$, $\cc(x_0^{-2}, \parent{x_0})$, and $\cc(x_n, x_{n-1})$ are independent. Notice also that 
$$ \bbP(E_n) \le \bbP(E_1)^{n/2 -1},$$
{as the  random variables 
$$ \left(\min_{y \in  \Hcal_n}  \frac {1 }{A(y)\beta_{{\omega}}(y)} \right)_{n \in \N} $$
are one-dependent.}

{Hence, we can choose  $C$  such  that  $\E[B(n)^p] < 2^{-pn}$,  which implies that 
$$  \bbE\left[\left(\frac 1{\beta_{{\omega}}(x_0)}\right)^{p} \right] \le C_{p} \sum_{n \ge 0} \E\left[2^{n(p-1)}B(n)^{p}\right] < \infty.
$$}



\end{proof}

Denote the number of visits to $\emp$ as
$$ L(\emp) := \sum_{j=0}^\infty \1_{X_j = \emp}.$$
\begin{lem}\label{Lfinmom}
If for some $p\ge 1$, $\bbE[\tilde \cc^p] <\infty$ and $\bbE[(\tilde \cc)^{-p}] <\infty$, then
			$$ \E\left[\left(\widetilde{L}(\emp)\right)^{p}\right]<\ff.$$
		\end{lem}
\begin{proof}
	Denote  by $z_1, z_2, \ldots$ the offspring of $\emp$. Denote by $Z^{\ssup{\ge}}$ the number of
$z_i$ satisfying $\cc( \emp, z_i)>\epsilon$, for some $\epsilon>0$ small enough.  Notice that $Z^{\ssup{\ge}}$ has all moments finite since it has a Poisson distribution with parameter $\Pi(\eps,\ff)$.

  Denote by $Z^{\ssup{\le}}$ the  number of vertices $z_i$ satisfying  $\cc( \emp, z_i) \le \epsilon$ and which are visited by the process $\mathbf{X}$.
	 We next argue that $Z^{\ssup{\le}}$ has all moments finite as follows. Each time a process visits a previously unvisited offspring of $\emp$ and the conductance assigned to the edge connecting this vertex to $\emp$ is less than $\epsilon$, then the annealed conditional probability of not returning to $\emp$, conditioned on the past, is bounded below by a fixed positive constant.
Now, fix one of the offspring of $\emp$, say $z$, and construct a tree $\L$ by taking $\emp$, $z$, their common edge, and the tree consisting of all the descendants of $z$. As before, consider the restriction of the process $\mathbf{X}$ to $\L$ and denote this restriction by $\mathbf{X}^{(\L)}$. We emphasize the fact that $\mathbf{X}$ and $\mathbf{X}^{(\L)}$ can be coupled in such a way that their steps  coincide up to the random time when  $\mathbf{X}$ leaves $\L$ forever. Under the quenched measure, the number of visits of $\mathbf{X}^{(\L)}$ to $\emp$, say $\phi(z)$, is geometrically distributed with mean  $ 1/\beta_{\omega}(z)$. 

The number of visits of $\mathbf{X}$ to $\emp$ is bounded by the sum of $\phi(z)$ over the $z$ which are the distinct offspring of $\emp$ that are visited by  $\mathbf{X}$.
	It is immediate, using exchangeability, to see  that 

	\begin{equation}\label{import}
	\bbE[L(\emp)^{p}] \le \bbE\left[\left(Z^{\ssup{\ge}}\right)^{p}\right] \bbE\left[\frac1{\beta^{p}_{\omega}(z_1)} \;\Big{|}\;\cc(z_1, \emp) > \epsilon\right]  + \bbE\left[\left(Z^{\ssup{\le}}\right)^{p}\right] \bbE\left[\frac1{\beta^{p}_{\omega}(z_1)} \;\Big{|}\;\cc(z_1, \emp) \le \epsilon\right] < \infty,
	\end{equation}
	where the finiteness of $\bbE[ 1/\beta^{p}_{\omega}(z_1) \;| \;\cc(z_1, \emp) > \epsilon]$, $\bbE[ 1/\beta^{p}_{\omega}(z_1) \;| \;\cc(z_1, \emp) < \epsilon]$ follows from Proposition~\ref{impro}.
\end{proof}

 Define
 $$ G_n := \#\{ v \colon |v| \le n,\; \exists\; j \in \N\cup \{0\} \mbox{ such that } X_j = v\},$$
 i.e., the number of vertices at distance less or equal to $n$ from the root, that are visited by the process.
\begin{lem}\label{lemcn} For some positive constant $K$, we have
$$ \bbE[G_n] \le K n.$$
\end{lem}
\begin{proof}
We fix $M>0$ large enough, to be specified later. Denote by
$$\mathcal{G}_n :=   \Big\{v \colon |v| \le n,\; \exists\; j \mbox{ such that } X_j = v,\; \cc(\parent{v}, v) < M\Big\}$$
Denote by $\widetilde{G}_n$ the cardinality of $\mathcal{G}_n$. Our goal (\eqref{inimpo} below) is to   provide a bound on  $G_n$ in terms of $\widetilde{G}_n$.  Order the distinct elements of $\mathcal{G}_n$, using, for example, the chronological order as they are visited by $\mathbf{X}$.  For each element of $\mathcal{G}_n$ we build a Galton-Watson tree,  as follows.
Fix the $i$-th element of $\mathcal{G}_n$. Call this  random vertex $v$.  Label the vertex $v$  with $i$   and label in the same way some of its descendants  according to the following rule. A vertex $u$, descendant of $v$, is labelled $i$  if 1) $\parent{u}$ is as well labelled $i$ and 2) $\cc(u, \parent{u}) \ge M$. Notice that,  if we choose $M$ large enough, the set of  vertices labelled $ i $  forms a subcritical Galton-Watson tree $\mathcal{T}_i$.  Denote by $Z_i$ the cardinality of this Galton-Watson tree.  Repeat this  procedure for each vertex in $\mathcal{G}_n$ and for the root of the tree.
We prove  the inequality
\begin{equation}\label{inimpo}
 G_n \le  \sum_{i=0}^{\widetilde{G}_n} Z_i.
 \end{equation}
To see that \eqref{inimpo} is true, it is enough to notice  that each vertex $v$ visited by the process, with $|v|\le n$ either satisfies
\begin{itemize}
\item $\cc(v, \parent{v}) <M$ and belongs to $\mathcal{G}_n$ (hence is counted in the right hand side of \eqref{inimpo} as a root of one of the $\mathcal{T}_i$, with  $i \le \widetilde{G}_n$); or
\item $\cc(v, \parent{v})  \ge M$ and $v$ is   a  vertex
 of one of the  $\mathcal{T}_i$, with  $i \le \widetilde{G}_n$.
\end{itemize}

Notice that $\bbE[Z_i] <\infty$  as this is the average size of a subcritical Galton-Watson tree.
Moreover, the $(Z_i)_i$ are identically distributed.
Hence, using \eqref{inimpo}, we have
\begin{equation}\label{infg}
  \bbE[G_n] \le \sum_{i=0}^{\infty} \bbE[Z_i \1_{\widetilde{G}_n \ge i}] =  \bbE[Z_1] \bbE[\widetilde{G}_n],
\end{equation}
{where in the last step we used independence between $\{\widetilde{G}_n \ge i\}$ and $Z_i$, which is a direct consequence of the definition of $Z_i$.}
Next, we prove that $\bbE[\widetilde{G}_n] = O(n)$ and this will end the proof of this lemma. To this end, denote by $g_i$ the number of vertices $v$  at level $i$   visited by the process and satisfying $\cc(v, \parent{v})<M$. We will see that $g_i$ is bounded by an exponential, with a parameter not depending on $i$.  Suppose that $x$ is the $j$-th vertex at level $i$ visited by the process and satisfying $\cc(x, \parent{x})<M$. 
We finish by proving
\begin{equation}\label{thepdn}
  \bbP(g_i > j+1\;\big|\; g_i > j) \le  \alpha,
  \end{equation}
for some $\alpha \in (0,1)$, independent of $i$ and $j$.  In order to prove \eqref{thepdn}, we reason as follows. {Given $g_i >  j$,  consider what happens at the first time  $\mathbf{X}$ visits the $(j+1)$-th distinct vertex, say $x$, at level $i$  with $\cc(x, \parent{x}) \le M$.}   If  it never goes back to $\parent{x}$ then $g_i = j+1$. Hence 
{
\begin{align*} \bbP(g_{i} = j+1\;\big|\; g_i > j)  &\ge \bbP( T_{\parent{x}}= \infty\;\big|\;  X_0 =x, \cc(x, \parent{x}) \le M) \\ &\ge \bbP( T_{\parent{x}}= \infty\;\big|\;  X_0 =x, \cc(x, \parent{x})=M) \ge 1 - \alpha,
\end{align*}
}
for some $\alpha>0$, where we use that $T_{\parent{x}}$ is monotone in the value of $\cc(x, \parent{x})$.
  Finally,
$$ \bbE[\widetilde{G}_n] = \sum_{i=1}^n \bbE[g_i] =  O(n).$$
\end{proof}

Set, for $v \neq \emp$,
\begin{equation}\label{keydef}
 \ L(v) := \sum_{j=0}^\infty \1_{X_j = v}\1_{X_{j+1} \neq \parent{v}},
\end{equation}
i.e., the number of times the process $\mathbf{X}$ jumps from $v$ to one of its offspring.
Denote by $\mathcal{T}(v)$ the subtree consisting of $v$, all its descendants, and the edges connecting them.
\begin{lem}\label{lem:tL}
There exists a random variable $\widetilde{L}(v)$ which is $\sigma(\{\cc(x, \parent{x}) \colon x \in \mathcal{T}(v), x \neq v\})$ measurable, such that
 \begin{itemize}
 \item $ L(v) \le \widetilde{L}(v)$, and
 \item $\widetilde{L}(v) $ has the same distribution as $L(\emp)$.
 \end{itemize}
\end{lem}

\begin{proof}
We construct a process $(\widetilde{X}_j(v))_{j\in\N}$ on $\mathcal{T}(v)$, Markovian under the quenched measure,  which uses the   conductances $\cc(x, \parent{x})$, with $ x \in \mathcal{T}(v)$ and $x \neq v$. This process is coupled with $\bf{X}$ as follows. 
Denote by $m_1$ the first time $\bf{X}$ hits  $v$, and recursively define
$$m_{i+1} := \inf\{ j > m_i \colon X_j \in \mathcal{T}(v), X_j \neq X_{m_i}\}.$$
Given the conductances, the process $(X_{m_i})_{i\in\N}$, is a Markov chain on $\mathcal{T}(v)$ up to a random time (possibly infinite) when the process $\mathbf{X}$ leaves $\mathcal{T}(v)$ for good.
We choose $(\widetilde{X}_j(v))_{j\in\N}$ to be a Markov chain (under the quenched meausure) which  satisfies
$$ \widetilde{X}_i(v) = X_{m_i},$$
for all $i$ such that $m_i<\infty$ (and for all $i$ such that $m_i=\ff$, it runs independently).
In words $(\widetilde{X}_i)_i $ is perfectly coupled with $(X_{m_i})_i$ up to the time the latter leaves $\mathcal{T}(v)$ for good.
If we set
\begin{eqnarray}\label{def:wtL}
\widetilde{L}(v) := \sum_{j=0}^\infty \1_{\widetilde{X}_j(v) = v}
\end{eqnarray}
it has the advertised properties.
\end{proof}


\noindent{\bf Proof of Theorem~\ref{finsecm}.}
We reason by contradiction.  Assume that
\begin{equation}\label{contr1}
\lim_{n \ti} \frac{|X_n|}n = 0, \qquad \mbox{a.s.} 
\end{equation}
{Of course \eqref{contr1}  would imply that }
\begin{equation}\label{contr1.5}
\lim_{n \ti} \frac{T(n)}{n} = +\infty, \qquad \mbox{a.s.,}
\end{equation}
{which in turn implies that}
\begin{equation}\label{contr2}
\bbE\left[\lim_{n \ti} \frac{T(n)}{n}\right] = \infty.
\end{equation}

On the other hand we prove, next, that there exists a constant $K$, such that 
\begin{equation}\label{contr3}
\bbE\left[\frac{T(n)}{n}\right] \le K, \qquad \mbox { for all $n$}.
\end{equation}
Once we have proved \eqref{contr3} we get a contradiction with \eqref{contr2} via Fatou's Lemma (the steps are shown below).
Order the distinct vertices $v_i$, $i \in \{1, 2, \ldots, G_n\}$,  visited by the process  such that $|v_i| \le n$, chronologically. Using the first property of $\tilde{L}$ in Lemma \ref{lem:tL} (see \eqref{def:wtL}), we bound $T(n)$ as follows:
$$ T(n) \le 2\sum_{i=1}^{G_n} (L(v_i) + L(\parent{v_i}) )  \le 2\sum_{i=1}^{G_n} (\tilde{L}(v_i) +\tilde{L}(\parent{v_i}) ) \le   4 \sum_{i=1}^{G_n} \tilde{L}(v_i) =  4 \sum_{i=1}^{\infty} \tilde{L}(v_i)\1_{G_n \ge i}. $$
Notice that $\1_{G_n \ge i}$  and $\tilde{L}(v_i)$ are independent.  In fact, the random variable $\tilde{L}(v_i)$ is independent  of the event weather  $v_i$ is visited or not by $\bf{X}$.  Hence for some finite constant  $K$
$$
\bbE\left[\sum_{i=1}^{\infty}(1+ \tilde{L}(v_i))\1_{G_n \ge i}\right] =
\sum_{i=1}^{\infty} \bbE\left[(1+ \tilde{L}(v_i)\right] \bbP(G_n \ge i) =\bbE\left[G_n\right] +  \bbE\left[\tilde{L}(\emp)\right]\bbE\left[G_n\right]\le K n,
$$
by virtue of Lemmas  \ref{Lfinmom}, \ref{lemcn} and \ref{lem:tL}, proving \eqref{contr3}. 
Finally, by Fatou's Lemma, we have
$$ K \ge \liminf_{n \ti} \bbE\left[ \frac{T(n)}n  \right] \ge \bbE\left[ \liminf_{n \ti}  \frac{T(n)}n  \right] = \infty,$$
yielding a contradiction and proving Theorem ~\ref{finsecm}.
\subsection{Proof of Theorem \ref{Hittimes}}

Fix $ q< p$.   Recall the definition of $\widetilde{L}(v)$ in \eqref{def:wtL} and the definition  of $(g_{i})_{i}$  given in the proof of Lemma~\ref{lemcn} .  
Label $\mu_i  k $, with $i \in \N$  and $k \le g_i$,  the vertices at level $i$ visited by the process and with the property $\cc(\mu_i k, \parent {(\mu_i k)}) < M$, for some fixed parameter $M$ as described in the proof of Lemma~\ref{lemcn} . Recall also the definition of $Z_{\mu_i  k}$ being the  size of the subcritical Galton--Watson subtree $\mathcal{G}(\mu_i  k)$ rooted at $\mu_i k$ and composed by vertices connected by edges whose conductances are larger than $M$.  In the sequel, for simplicity, we drop the subscript from $\mu_i k$. Define 
$$ D_{\mu k}  = \sum_{v \in \mathcal{G}(\mu_i  k)}  \widetilde{L}(v).$$
Next, we prove that $\E[D_{\mu k}^q]<\infty,$ for all $q<p$.  In fact, using Jensen and Holders\rq{} inequalities, we have
$$
\begin{aligned}
 \E\left[ \left(\sum_{v \in \mathcal{G}(\mu)}  \widetilde{L}(v)\right)^q\right] &\le  \E\left[ Z^{q-1}_\mu\left(\sum_{v \in \mathcal{G}(\mu)}  \widetilde{L}^q(v)\right)\right]\\
 &=  \E\left[ \sum_{v}  \widetilde{L}^q(v) Z^{q-1}_\mu \1_{v \in \mathcal{G}(\mu)}\right]\\
 &\le \E[  \widetilde{L}^p(v)]^{q/p} \sum_{v}  \E\left[Z^{(q-1)p/(p-q)}_\mu \1_{v \in \mathcal{G}(\mu)}\right]^{(p-q)/p}<\infty,
 \end{aligned}
 $$
 where the finitess of the last expression is derived by the fact that $Z$ has geometric tail and Lemma~\ref{Lfinmom}. Next, let $q<p^*<p$, 
\begin{equation}
\begin{aligned}
  \E\left[(T(n))^q\right] &\le  4^q 
\E \left[ \left(\sum_{k=1}^{G_n}  \widetilde{L}(v_{k})  \right)^q\right]\\
&\le 4^q \E \left[ \sum_{i=1}^{n} \left(\sum_{k=1}^{g_{i}} D_{\mu k}\right)^{q}\right]\\
&\le n^{q-1} n 4^q \E\left[\left(\sum_{k=1}^{g_{1}} D_{\mu k}\right)^{q}\right]\\
&= n^{q} 4^q \E\left[ g^{q-1}_{1} \sum_{k=1}^{\infty}  D_{\mu k}^{q}  \1_{g_1 \ge k}\right]\\
&\le  n^{q} 4^q \E\left[ D_{\mu 1}^{p^*}\right]^{q/p^*} \sum_{k=1}^{\infty}\E\left[ g^{(q-1)p^*/(p^*-q)}_{1} \1_{g_1 \ge k}\right]^{(p^*-q)/p^*} <\infty.
\end{aligned}
\end{equation}
Hence the collection of random variables $(T^{b}(n)/n^{b})_{n}$ is uniformly integrable for each $b  <p$. We already proved that under more general conditions $T(n)/n$ converges a.s. to $1/s$, and this yield our results.

 \subsection{Proof of Theorem \ref{thinf}}

 \begin{lem}\label{mainlem} For any  $M$ large enough, there exists a.s. a random vertex  $v$,  such  that  
 \begin{eqnarray*} \label{conM1}\cc(v, \parent{v})<M
 \end{eqnarray*}
 and such that $X_k=v$, for some $k\in\N$.
   \end{lem}
   \begin{proof}
   Consider the Galton Watson tree $\mathcal{T}_\emp$ rooted at $\emp$ and consisting of descendant vertices connected to $\emp$ only by edges with conductances larger than $M$. Such a  tree was used in the proof of Lemma~\ref{lemcn}.   If $M$ is large enough, then $\mathcal{T}_\emp$ is a.s. finite and so the process $\mathbf{X}$ will hit, in a.s. finite time, a vertex $v$ such that $\cc(v, \parent{v})<M$.
   \end{proof}
  
 Fix $M>0$ and choose a $v$ as in Lemma \ref{mainlem} and note that $T_v$ is a stopping time. Define
 $$ S_v := \inf\{ n > T_{v} \colon  |X_n - v| =  2\}.$$
 Also, let  $\{v_i, i \in \N\}$ be the offspring of $v$, ordered in such a way that $\cc(v, v_i) \ge \cc(v, v_{i+1})$. The following result shows that the version of our model with the largest conductance having infinite mean is similar to a trap model (see \cite[Section 5]{benarous2006}). {Denote by $\Fcal_{T_v}$ the $\sigma$-algebra 
 consisting of the sets $A \in \Fcal_{\infty}$ such that $A \cap \{T_v = k \} \in \Fcal_k$ for $k \in \N$. 
   \begin{prop}[Large conductances are traps] There is a constant $c>0$ such that
$$ 	\bbP(S_v \ge n\;|\;\Fcal_{T_v}) \ge c \bbP(\cc(v, v_1)\ge n).$$ In particular, when $\bbE[\cc(v, v_1)]=\ff$ then 
$$ 	\bbE[S_v] = \infty. $$	
 	\end{prop}
 \begin{proof}
 
   Let $\{d_i, i \in \N\}$ be the offspring of $ v_1$ and set 
  $$ D :=  \left\{
 \sum_{i=2}^\infty \cc(v, v_i)<M, \qquad \sum_{i=1}^\infty\cc(v_1, d_i)<M \right\}.$$ 
 {Fix  an event $A$ such that $A \cap \{T_v = t\} \in \Fcal_t$. The event $A$ must be independent of $D$, as the conductances involved in the definition of $D$ can be used only after time $T_v$.
     In fact, both $D$ and $\CC(v, v_1)$ are independent of  $\Fcal_{T_v}$, and $\bbP(D) > 0$. Thus}
 $$
 \begin{aligned}
   \bbP(S_v \ge n, D\;|\; \cc(v, v_1), \Fcal_{T_v}) &\ge \prod_{i=1}^n \frac{\cc(v, v_1)}{\cc(v, v_1) +2M} \bbP(D) = \left(1 - \frac{2M}{\cc(v, v_1) +2M}\right)^n \bbP(D).
   \end{aligned}
   $$
  We therefore have 
   $$
   \begin{aligned}
   \bbP(S_v \ge n\;|\;\Fcal_{T_v})  &\ge \bbP(S_v \ge n \;|\; \cc(v, v_1)\ge n, \Fcal_{T_v})  \bbP(\cc(v,  v_1)\ge n\;|\;\Fcal_{T_v})\\
   &\ge \bbP(S_v \ge n, D \;|\; \cc(v, v_1)\ge n, \Fcal_{T_v})  \bbP(\cc(v,  v_1)\ge n\;|\;\Fcal_{T_v})\\
   &\ge \left(1 - \frac{2M}{n +2M}\right)^n \bbP(D)\bbP(\cc(v, v_1)\ge n\;|\;\Fcal_{T_v})\\
   &= \left(1 - \frac{2M}{n +2M}\right)^n \bbP(D) \bbP(\cc(v, v_1)\ge n)\\
   &\ge {\rm e}^{-2M + o(1)}\bbP(D)\bbP(\cc(v, v_1)\ge n).
   \end{aligned}
   $$
   We conclude that
   $$ \bbE[S_v\;|\;\Fcal_{T_v}] = \sum_{n=1}^\infty  \bbP(S_v \ge n\;|\;\Fcal_{T_v})  \ge  \sum_{n=1}^\infty {\rm e}^{-2M + o(1)} \bbP(\cc(v, v_1)\ge n\;|\;\Fcal_{T_v}) = \infty,$$
   where we used the fact, under the infinite mean assumption,
   \begin{equation}\label{infsum}
   \sum_{n=1}^\infty \bbP({\cc(v, v_1)} \ge n\;|\;\Fcal_{T_v})=    \sum_{n=1}^\infty \bbP({\cc(v, v_1)} \ge n)= \infty. \qquad 
   \end{equation}
   

 \end{proof}

\noindent{\bf Proof of Theorem~\ref{thinf}}.

As a consequence of transience, there exists an infinite sequence of regenerative times, which can be described as the hitting times of levels which are visited exactly once. Let
$$\tau_1 := \inf\{k\colon |X_k| < |X_u|\; \mbox{ for all  $u > k$} \}.$$
Define, recursively,
\begin{align}\label{tau def}
\tau_i := \inf\{k > \tau_{i-1}\colon |X_k| <  |X_u|\; \mbox{ for all  $u > k$} \}.
\end{align} 
The sequence $(\tau_i)_i$ is the { regenerative-time} sequence.
Define also the regenerative levels $\ell_i = |X_{\tau_i}|$.

Choose $m$  large enough so that the expected regeneration-time interval satisfies \mbox{$\bbE[\tau_m - \tau_1] \ge  \bbE[T(5)]$.}
   By our choice of $m$, we have that the regeneration-time interval $\tau_m - \tau_1$  has infinite first moment since
   $$
  \begin{aligned}
     \bbE[\tau_m - \tau_1] &\ge  \bbE[T(5)]  \ge \bbE[T(5) \1_{T_{v} < T(3)}] \ge \bbE[S_v \1_{T_{v} < T(3)}] \\
     &\ge \bbE[  \1_{T_{v} < T(3)}\bbE[S_v\;|\;\Fcal_{T_v}]] = \infty  \cdot \bbP(T_{v} < T(3)) = \infty.
     \end{aligned}
     $$
   Recall that each $\tau_n$ coincides with a $T(\ell_n)$ for regeneration level $\ell_n$.  Hence
  $$  \limsup_{n \ti}  \frac{T(n)}n  \ge  \lim_{n \ti}  \frac{\sum_{k=1}^{\floor{\ell_n/m}} \tau_{m k} - \tau_{m(k-1)}}{\ell_n}   = \infty, \qquad \mbox{a.s.,}$$
via the Strong Law of Large Numbers. 
  Finally, we have
  $$ \liminf_{n \ti} \frac{|X_n|}n = \liminf_{n \ti} \frac n{T(n)}  = 0, \qquad \mbox{a.s.}$$\section*{Acknowledgements}
 We  thank Daniel Kious and Greg Markowsky for helpful discussions, and an anonymous referee for pointing out an error in an earlier version of the paper. A.C. was supported by ARC grant DP140100559. P. J. was supported in part by NSA grant H98230-14-1-0144 and NRF grant N01170220.

\bibliographystyle{alpha}
\bibliography{BibSep16}	

\end{document}